\documentclass[a4paper]{amsart}

\usepackage[T1]{fontenc}
\usepackage[utf8]{inputenc}
\usepackage{textcomp}
\usepackage{lmodern}
\usepackage{amsmath,amsthm,amssymb}
\usepackage{microtype}
\usepackage{graphicx}
\usepackage{todonotes}
\usepackage{commath}
\usepackage[numbers,sort,compress]{natbib}
\usepackage{hyperref}
\hypersetup{%
  pdfauthor={Matthias. J. Weber, Hans-Peter Schröcker},%
  pdftitle={Minimal area conics in the elliptic plane},%
  pdfkeywords={minimal conic, sphero-conic, spherical ellipse, covering cone, elliptic geometry, spherical geometry, enclosing conic, uniqueness},%
  colorlinks=true,%
  citecolor=black,%
  urlcolor=black,%
  linkcolor=black,%
}

\newcommand*{\R}{\mathbb{R}}

\newcommand*{\Part}[2]{\frac{\partial #1}{\partial #2}}
\newcommand*{\tp}{\mathrm{T}}
\newcommand*{\lSet}{\mathcal{L}}
\newcommand*{\SO}[1][3]{\mathrm{SO}(#1)}
\newcommand*{\ellE}{\bar{E}}
\newcommand*{\ellK}{\bar{K}}
\newcommand*{\Rho}{\mathrm{R}}
\DeclareMathOperator{\diag}{diag}

\DeclareMathOperator{\area}{area}

\newtheorem{theorem}{Theorem}
\newtheorem{lemma}[theorem]{Lemma}
\newtheorem{proposition}[theorem]{Proposition}
\newtheorem{corollary}[theorem]{Corollary}
\theoremstyle{definition}
\newtheorem{definition}[theorem]{Definition}
\newtheorem{example}[theorem]{Example}
\theoremstyle{remark}

\begin{document}

\title{Minimal area conics in the elliptic plane}
\author{Matthias.\,J.~Weber \and Hans-Peter~Schröcker}
\date{\today}

\address{Matthias.\,J.~Weber, Hans-Peter Schröcker, Unit Geometry and
  CAD, University Innsbruck, Technikerstraße 13, 6020 Innsbruck,
  Austria}


\begin{abstract}
  We prove some uniqueness results for conics of minimal area that
  enclose a compact, full-dimensional subset of the elliptic
  plane. The minimal enclosing conic is unique if its center or axes
  are prescribed. Moreover, we provide sufficient conditions on the
  enclosed set that guarantee uniqueness without restrictions on the
  enclosing conics. Similar results are formulated for minimal
  enclosing conics of line sets as well.
\end{abstract}

\maketitle

\section{Introduction}
\label{sec:introduction}

It is well know that every compact subset $F$ of the Euclidean plane
with inner points can be enclosed by a unique ellipse of minimal
area. More generally, every compact subset $F$ of $d$-dimensional
Euclidean space with inner points defines a unique enclosing ellipsoid
of minimal volume (see \cite{behrend37:_affininv_konvex,%
  behrend38:_kleinste_ellipse} for the case $d=2$ and
\cite{john48:_studies_and_essays,%
  danzer57:_loewner_ellipsoid} for general~$d$).

These uniqueness results are generally considered as ``easy''. The
reason for the existence of simple proofs are illuminated by recent
publications of the authors
\cite{schroecker08:_uniqueness_results_ellipsoids,%
  weber10:_davis_convexity_theorem}. We showed that numerous
uniqueness results in Euclidean spaces are a consequence of a simple
convexity property of the function that measures the ellipsoid's
size. Most notably, minimal enclosing ellipsoids with respect to
quermass integrals are unique (see also
\cite{firey64:_means_of_convex_bodies,gruber08:_john_type}).

In the present article we establish first uniqueness results in the
elliptic plane. We provide sufficient conditions on the enclosed set
$F$ that guarantee uniqueness of the minimal enclosing conic. To the
best of our knowledge, these are the first uniqueness results in a
non-Euclidean geometry. This is maybe the case because our uniqueness
results are not easy in the sense described above. While uniqueness in
the co-axial and concentric case can still be deduced from a convexity
property of the area function, uniqueness in the general case requires
extra work. The necessary calculations are rather involved and
constitute the largest part of this article.

We mostly use the spherical model of the elliptic plane. It is easily
obtained from the bundle model whose ``points'' are the
one-dimensional subspaces of the vector space $\R^3$. The distance of
two points in the bundle model is defined as the Euclidean angle
between lines and is a bi-valued function. The straight lines in the
bundle model are the two-dimensional subspaces. Their angle is the
usual Euclidean angle.

In this setting, computational aspects of the minimal circular cone
problem (where uniqueness is elementary) have already attracted the
attention of applied mathematicians
\cite{lawson65:_smallest_covering_cone,%
  barequet05:_optimal_bounding_cones}. We believe that the
applications mentioned in \citep{barequet05:_optimal_bounding_cones}
could profit from using minimal enclosing conics (or cones of second
degree) instead of circles (or right circular cones).

The spherical model of the elliptic plane is obtained by intersecting
the bundle model with the unit sphere $S^2$. The metric is inherited
from the ambient Euclidean space and the only difference to spherical
geometry is the identification of antipodal points. This is merely a
technical issue so that our uniqueness results can also be formulated
for sphero-conics.

We continue this article by an introduction to conics in the elliptic
plane. In Section~\ref{sec:area-conic} we derive some convexity
properties of their area function which are used in
Section~\ref{sec:coaxial-and-concentric-conics} for proving uniqueness
in the co-axial and concentric case. The general uniqueness result is
presented in Section~\ref{sec:general-case}. In
Section~\ref{sec:line-sets} we derive uniqueness results for minimal
enclosing conics of line-sets, analogous to those of
\cite{schroecker07:_minim_hyper}. The duality between points and lines
in the elliptic plane makes them simple corollaries. Most auxiliary
results are collected in an appendix.

\section{Preliminaries}
\label{sec:preliminaries}

A conic $C$ in the spherical model of the elliptic plane is the
intersection of the unit sphere $S^2$ with a quadratic cone whose
vertex is the center of~$S^2$:
\begin{equation}
  \label{eq:1}
  C = \{x \in S^2\colon x^\tp \cdot M \cdot x = 0\},
\end{equation}
where $M \in \R^{3 \times 3}$ is an indefinite symmetric matrix of
full rank. Since proportional matrices describe the same conic, it is
no loss of generality to assume that $M$ has eigenvalues $\nu_1 \ge
\nu_2 > 0$ and $\nu_3 = -1$. Then the interior of the conic $C$
consists of all points $x$ that fulfill the inequality $x^\tp \cdot M
\cdot x < 0$.

\begin{figure}
  \centering
  \includegraphics{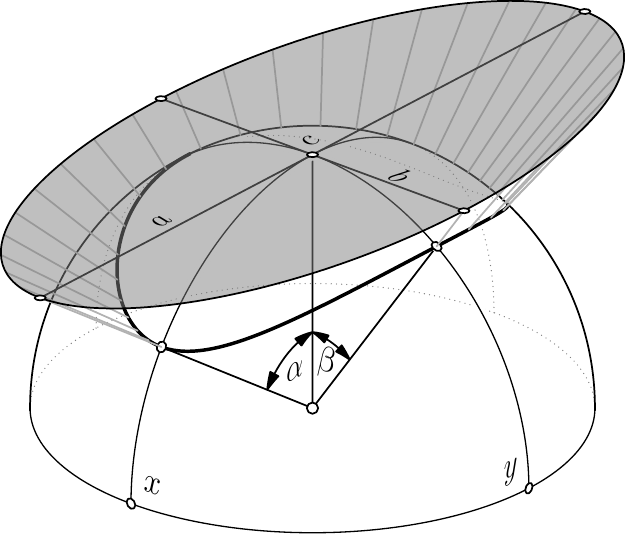}
  \caption{Conic as intersection of $S^2$ with a quadratic cone}
  \label{fig:sphero-conic}
\end{figure}

The transformation group of elliptic geometry is the rotation group
$\SO$. Thus, the matrix $M$ has the normal form
\begin{equation}
  \label{eq:2}
  M = \diag(\nu_1, \nu_2, -1) = \diag(b^{-2}, a^{-2}, -1)
\end{equation}
where $a = \nu_2^{-1/2}$, $b = \nu_1^{-1/2}$. The values $\alpha =
\arctan a$ and $\beta = \arctan b$ are the conic's \emph{semi-axis
  lengths} (Figure~\ref{fig:sphero-conic}).

Generally, the three points $x = (1,0,0)^\tp$, $y = (0,1,0)^\tp$, and
$c = (0,0,1)^\tp$ are called the centers of $C$. Any line through one
of them is a diameter since its intersection points with the conic are
at equal (possibly complex) distance to the center. Among the three
centers the point $c$ is distinguished by the fact that it lies in the
interior of $C$. It has also a special meaning in the context of
minimal enclosing conics so that we use the word \emph{center}
exclusively for the point $c$. The three lines spanned by any two of
the points $x$, $y$, and $c$ are lines of symmetry and are often
called \emph{axes}. We reserve this term for the two lines $X = c \vee
x$ and $Y = c \vee y$. If $\alpha > \beta$, we call $X$ the
\emph{major axis} and $Y$ the \emph{minor axis.}

Note that elliptic geometry does not distinguish between different
types of regular conics apart from circles and non-circular (general)
conics. Indeed, the point set \eqref{eq:1} satisfies the well-known
focal definitions of both, ellipses and hyperbolas. This is possible
because the distance in the elliptic plane is bi-valued.

If $M$ is not in normal form, semi-axis lengths, center, and axes can
be obtained from the eigenvalues and eigenvectors of $M$.  Denote the
vector of eigenvalues of $M$, arranged in decreasing order, by $e(M)$,
\begin{equation}
  \label{eq:3}
  e(M) = (\nu_1, \nu_2, \nu_3)^\tp, \quad\text{where}\quad
  \nu_1 \ge \nu_2 > 0 \quad\text{and}\quad \nu_3 < 0,
\end{equation}
and the corresponding eigenvectors by $y$, $x$, and $c$. The function
\begin{equation}
  \label{eq:4}
  w \colon (\nu_1, \nu_2, \nu_3)^\tp \mapsto
  (a, b)^\tp =
  \Bigl( \sqrt{-\frac{\nu_3}{\nu_2}}, \sqrt{-\frac{\nu_3}{\nu_1}} \Bigr)^\tp,
\end{equation}
computes the tangents $a = \tan \alpha$ and $b = \tan \beta$ of the
semi-axis lengths $\alpha$ and $\beta$. The vector $c$ points to the
center of $C$. If $\nu_1 > \nu_2$, the conic is not a circle, the
major axis exists and is incident with~$x$.

\section{The area of a conic in the elliptic plane}
\label{sec:area-conic}

Now we derive some properties of the area function of conics in the
elliptic plane. The surface area of $C$ is a strictly monotone
increasing function of $\alpha$ and $\beta$. For our purpose it is
more convenient to view it as a strictly monotone decreasing function
of the variables $-\nu_1/\nu_3$ and $-\nu_2/\nu_3$.

\subsection{The area function.}
\label{sec:area-function}

We compute the area for the normal form \eqref{eq:2} of the matrix
$M$. The upper half of the unit sphere $S^2$ can be parametrized as
\begin{equation}
  \label{eq:5}
  S^2 \colon
  \begin{pmatrix}
    \cos \varphi \cos \vartheta\\
    \sin \varphi \cos \vartheta\\
    \sin \vartheta
  \end{pmatrix}, \quad
  \varphi \in [-\pi, \pi], \ \vartheta \in \bigl[ 0, \frac{\pi}{2} \bigr].
\end{equation}
The points inside $C$ belong to parameter values $(\varphi,
\vartheta)$ related by
\begin{equation}
  \label{eq:6}
  \vert \vartheta \vert > \vartheta_0 =
  \arcsin\biggl(\sqrt{\frac{a^2
      \sin^2 \varphi + b^2 \cos^2 \varphi}{a^2 b^2 + a^2
      \sin^2 \varphi + b^2 \cos^2 \varphi}}\biggr).
\end{equation}
By integrating the area element $\cos\vartheta\, \dif\vartheta \wedge
\dif\varphi$ of \eqref{eq:5} we obtain the area of the conic $C$ as
\begin{multline}
  \label{eq:7}
  \area(C) = \area(a,b) =
  \int_{-\pi}^{\pi} \int_{\vartheta_o}^{\pi/2} \cos\vartheta \dif\vartheta \dif\varphi =
  \int_{-\pi}^{\pi} 1 - \sin\vartheta_0 \dif\varphi \\
  = 2 \pi - \int_{-\pi}^{\pi} \sqrt{\frac{a^2
    \sin^2 \varphi + b^2 \cos^2 \varphi}{a^2 b^2 + a^2
    \sin^2 \varphi + b^2 \cos^2 \varphi}} \dif\varphi.
\end{multline}
The integral representation \eqref{eq:7} is perfectly suitable for our
purposes so that we refrain from expressing the area in terms of
elliptic integrals. Substituting $a^2 = -\nu_3/\nu_2$ and $b^2 =
-\nu_3/\nu_1$ into \eqref{eq:7} we obtain the area in terms of the
eigenvalues of the matrix $M$:
\begin{equation}
  \label{eq:8}
    \area(C) = \area(\nu_1, \nu_2, \nu_3)
    = 2 \pi - \int_{-\pi}^{\pi} \sqrt{\frac{\nu_1 \sin^2 \varphi + \nu_2
        \cos^2 \varphi}{-\nu_3 + \nu_1 \sin^2 \varphi + \nu_2 \cos^2
        \varphi}} \dif\varphi.
\end{equation}
The matrix $M$ is determined by the conic $C$ only up to a scalar
factor. Thus, we may normalize it such that $\nu_3 = -1$. Then the
area becomes
\begin{equation}
  \label{eq:9}
  \area(C) = \area(\nu_1, \nu_2) =
  2 \pi - \int_{-\pi}^{\pi} \sqrt{\frac{\nu_1 \sin^2 \varphi + \nu_2 \cos^2 \varphi}{1 +
      \nu_1 \sin^2 \varphi + \nu_2 \cos^2 \varphi}} \dif\varphi.
\end{equation}

\subsection{Convexity of the area function}
\label{sec:area-convex}

We prove that the function \eqref{eq:9} is strictly convex for
$\nu_1$, $\nu_2 > 0$. The standard arguments of
\cite{schroecker08:_uniqueness_results_ellipsoids,%
  weber10:_davis_convexity_theorem} then imply uniqueness of the
minimal enclosing conic among all conics with prescribed axes or
prescribed center. These proofs are given later, in
Section~\ref{sec:uniqueness-results}.

\begin{lemma}
  \label{lem:1}
  The area function \eqref{eq:9} is strictly convex.
\end{lemma}

\begin{proof}
  We show that the Hessian $H$ of \eqref{eq:9} is positive definite,
  that is, all its principal minors are positive. The upper left
  entry of $H$ equals
  \begin{equation}
    \label{eq:10}
    \dpd[2]{\area}{\nu_1} =
    \frac{1}{4} \int_{-\pi}^{\pi} J \sin^4 \varphi \dif\varphi,
  \end{equation}
  where
  \begin{equation}
    \label{eq:11}
    J = \frac{1 + 4 \nu_1 \sin^2 \varphi + 4 \nu_2 \cos^2 \varphi}
             {(\nu_1 \sin^2 \varphi + \nu_2 \cos^2 \varphi)^{3/2}
              (1 + \nu_1 \sin^2 \varphi + \nu_2 \cos^2 \varphi)^{5/2}}.
  \end{equation}
  Clearly, the integral \eqref{eq:10} is strictly positive. The
  determinant of $H$ equals
  \begin{multline}
    \label{eq:12}
    \dpd[2]{\area}{\nu_1}\dpd[2]{\area}{\nu_2}-
    \biggl(\dmd{\area}{}{\nu_1}{}{\nu_2}{}\biggr)^2 = \\
    \frac{1}{16} \int_{-\pi}^{\pi} J \sin^4 \varphi
    \dif\varphi \cdot \int_{-\pi}^{\pi} J \cos^4 \varphi \dif\varphi -
    \frac{1}{16} \Bigl( \int_{-\pi}^{\pi} J \sin^2 \varphi
    \cos^2 \varphi \dif\varphi \Bigr)^2.
  \end{multline}
  By the Schwarz inequality we have
  \begin{equation}
    \label{eq:13}
    \sqrt{\int_{-\pi}^{\pi} \bigl( \sqrt{J} \sin^2 \varphi \bigr)^2 \dif\varphi}
    \cdot
    \sqrt{\int_{-\pi}^{\pi} \bigl( \sqrt{J} \cos^2 \varphi \bigr)^2 \dif\varphi}
    \ge \int_{-\pi}^{\pi} J \sin^2\varphi \cos^2\varphi \dif\varphi,
  \end{equation}
  with equality precisely if the integrands on the left are
  proportional. Since this is not the case, \eqref{eq:12} is strictly
  positive as well. Hence, the Hessian of $H$ is positive definite and
  $\area(\nu_1,\nu_2)$ is strictly convex.
\end{proof}

\section{Uniqueness results}
\label{sec:uniqueness-results}

We are aware of two essentially different methods for proving
uniqueness of minimal circumscribed (or maximal inscribed) conics. One
may consider the problem as an optimization task and derive sufficient
conditions for the existence of a unique minimizer or maximizer. This
is the approach of \cite{john48:_studies_and_essays,%
  juhnke90:_min_ellipsoid,%
  juhnke94:_semi_infinite}. The second method of proof is
indirect. Assuming existence of two minimizers (or maximizers in case
of inscribed conics) $C_0$ and $C_1$ one shows existence of a further
circumscribing (or inscribed) conic $C$ of smaller (or larger)
size. This idea or variants of it can be found in
\cite{danzer57:_loewner_ellipsoid,%
  klartag04:_john_type_ellipsoids,%
  gruber08:_john_type,%
  schroecker08:_uniqueness_results_ellipsoids,%
  weber10:_davis_convexity_theorem}. In this article, we adopt it as
well. In our setup, the equation of the conic $C$ is found as a convex
combination of the respective equations of $C_0$ and~$C_1$:

\begin{definition}[in-between conic]
  \label{def:1}
  Let $C_0$ and $C_1$ be two conics
  \begin{equation}
    \label{eq:14}
    C_i = \{ x \in S^2 \colon x^\tp \cdot M_i \cdot x = 0 \}, \quad i
    = 0, 1
  \end{equation}
  such that the matrices $M_i$ are indefinite and have precisely one
  negative eigenvalue $\nu_{3,i} = -1$. For $\lambda \in (0,1)$ we
  define the \emph{in-between conic} $C_\lambda$ to $C_0$ and $C_1$ as
  \begin{equation}
    \label{eq:15}
    C_\lambda = \{x \in S^2 \colon x^\tp \cdot M_\lambda \cdot x = 0\},
  \end{equation}
  with
  \begin{equation}
    \label{eq:16}
    M_\lambda = (1-\lambda) M_0 + \lambda M_1.
  \end{equation}
\end{definition}

We also use the symbolic notation $C_\lambda = (1-\lambda)C_0 +
\lambda C_1$. As long as $C_0$ and $C_1$ have a common interior,
$C_\lambda$ is a non-degenerate conic whose interior contains the
common interior of $C_0$ and $C_1$. In general, the unique negative
eigenvalue of $M_\lambda$ is different from $-1$ (in fact larger than
$-1$ as the smallest eigenvalue is a concave function of
$\lambda$). This hinders the usage of \eqref{eq:9} and accounts for
most difficulties in the general proof of uniqueness. If the conics
$C_0$ and $C_1$ have the same axes or the same center the situation is
much simpler.

\subsection{Coaxial and concentric conics}
\label{sec:coaxial-and-concentric-conics}

We call a subset $F$ of the elliptic plane bounded, if it is contained
in a circle and we call it full-dimensional if it is not contained in
a line. In this section we prove that any bounded, compact and
full-dimensional subset $F$ of the elliptic plane can be enclosed by a
\emph{unique} conic of minimal area with prescribed axes or center.
The proofs of uniqueness are simple and follow the general scheme
outlined in
\cite{schroecker08:_uniqueness_results_ellipsoids}. Nonetheless, the
concentric case constitutes the basis for the much deeper general
uniqueness result in Section~\ref{sec:general-case}.

\begin{theorem}
  \label{th:1}
  Let $F$ be a bounded, compact and full-dimensional subset of the
  elliptic plane. Among all conics with two given axes that contain
  $F$ there exists exactly one of minimal area.
\end{theorem}

\begin{proof}
  Existence is a direct consequence of compactness and boundedness of
  $F$ and continuity of the area function. In order to show
  uniqueness, assume $C_0$ and $C_1$ are two minimal conics with
  prescribed axes and circumscribing $F$. Because $F$ is
  full-dimensional, both $C_0$ and $C_1$ are not degenerate.  In a
  suitable coordinate frame we can describe them by diagonal matrices
  \begin{equation}
    \label{eq:17}
    M_i = \diag(\nu_{i,1}, \nu_{i,2}, -1),\quad
    \nu_{i,1} \ge \nu_{i,2} > 0,\quad
    i = 0,1.
  \end{equation}
  The in-between conic $C_\lambda$ is then given by
  \begin{equation}
    \label{eq:18}
    M_\lambda = \diag\bigl( (1-\lambda) \nu_{0,1} + \lambda \nu_{1,1},\
    (1-\lambda) \nu_{0,2} + \lambda \nu_{1,2}, -1 \bigr).
  \end{equation}
  Because the area function \eqref{eq:9} is strictly convex we have
  \begin{equation}
    \label{eq:19}
    \area \circ\, w \circ e(M_\lambda) <
    (1-\lambda) \area \circ\, w \circ e(M_0) + \lambda \area \circ\, w \circ e(M_1)
  \end{equation}
  (the functions $e$ and $w$ are defined in \eqref{eq:3} and
  \eqref{eq:4}, respectively).  Hence, the area of $C_\lambda$ is
  strictly smaller than that of $C_0$ and $C_1$\,---\,a contradiction
  to the assumed minimality of $C_0$ and~$C_1$.
\end{proof}

Uniqueness of minimal enclosing conics among all conics with
prescribed center follows again from the strict convexity of
\eqref{eq:9} and

\begin{proposition}[Davis' Convexity Theorem]
  A convex, lower semi-continuous and symmetric function $f$ of the
  eigenvalues of a symmetric matrix is (essentially strict) convex on
  the set of symmetric matrices if and only if its restriction to the
  set of diagonal matrices is (essentially strict) convex.
\end{proposition}

A proof for the convex case is given in
\cite{davis57:_convex_functions}. The extension to essentially strict
convexity is due to \cite{lewis96:_convex_analysis}. We skip the
technicalities related to the precise definition of ``essentially
strict convexity''. All prerequisites are met in our case
and all necessary conclusions can be drawn.

\begin{theorem}
  \label{th:2}
  Let $F$ be a bounded, compact and full-dimensional subset of the
  elliptic plane. Among all conics with given center that contain $F$
  there exists exactly one of minimal area.
\end{theorem}

\begin{proof}
  The proof is similar to that of Theorem~\ref{th:1}. Instead of the
  diagonal matrices \eqref{eq:17} and \eqref{eq:18} we have matrices
  of the shape
  \begin{equation}
    \label{eq:20}
    M_i =
    \begin{pmatrix}
      \star & \star & 0 \\
      \star & \star & 0 \\
      0     & 0     & -1
    \end{pmatrix},
    \quad i \in \{0,1,\lambda\}.
  \end{equation}
  Davis' Convexity Theorem guarantees strict convexity of the function
  $\area \circ\, w \circ e$ on the space of matrices of
  type~\eqref{eq:20}.
\end{proof}

\subsection{The general case}
\label{sec:general-case}

Now we come to the general case. Here, we cannot make use of Davis'
Convexity Theorem since the negative eigenvalue of the matrix
$M_\lambda$ is different from $-1$ and the area can no longer be
regarded as convex function in the positive eigenvalues of
$M_\lambda$. In fact, there exist situations where
\begin{equation}
  \label{eq:21}
  \area(C_\lambda) > \area(C_0) = \area(C_1)
\end{equation}
for all in-between conics $C_\lambda$. We present an example of this:

\begin{example}
  \label{ex:1}
  Let $C_0$ and $C_1$ be two congruent conics described by
  \begin{equation}
    \label{eq:22}
    \textstyle
    M_0 = \diag(\frac{1}{16}, \frac{1}{36}, -1)
    \quad\text{and}\quad
    M_1 = R_1 \cdot R_2 \cdot R_3 \cdot M_0 \cdot (R_1 \cdot R_2 \cdot R_3)^\tp
  \end{equation}
  where $R_1$, $R_2$, $R_3$ are the rotation matrices
  \begin{equation}
    \label{eq:23}
    \textstyle
    R_1 =
    \left(
    \begin{smallmatrix}
      1 & 0                 & 0                  \\
      0 & \cos\frac{\pi}{60} & -\sin\frac{\pi}{60} \\
      0 & \sin\frac{\pi}{60} & \phantom{-}\cos\frac{\pi}{60}
    \end{smallmatrix}\right),\
    R_2 =
    \left(
    \begin{smallmatrix}
      \cos\frac{\pi}{36} & 0 & -\sin\frac{\pi}{36} \\
      0                  & 1 & 0                   \\
      \sin\frac{\pi}{36} & 0 & \phantom{-}\cos\frac{\pi}{36}
    \end{smallmatrix}\right),\
    R_3 =
    \left(
    \begin{smallmatrix}
      \cos\frac{\pi}{6} & -\sin\frac{\pi}{6}           & 0 \\
      \sin\frac{\pi}{6} & \phantom{-}\cos\frac{\pi}{6} & 0 \\
      0                 & 0                            & 1
    \end{smallmatrix}\right).
  \end{equation}
  The two conics are congruent (hence of equal area) and have a
  non-empty common interior. Figure~\ref{fig:sph_CounterExample},
  left, displays them together with a few in-between
  conics. Figure~\ref{fig:sph_CounterExample}, right, shows a plot of
  the area function of $C_\lambda$ on the interval $(0,1)$. We see
  that the area of any in-between conic is larger than $\area C_0 =
  \area C_1$.
\end{example}

Example~\ref{ex:1} illustrates the difficulties we have to expect when
proving uniqueness results for non concentric conics in the elliptic
plane. Additional assumptions on the enclosed set $F$ are inevitable
at least for our method of proof which is based on in-between conics
of Definition~\ref{def:1}.

\begin{figure}
  \begin{minipage}[b]{0.5\linewidth}
    \rule{\linewidth}{0pt}
    \centering
    \includegraphics{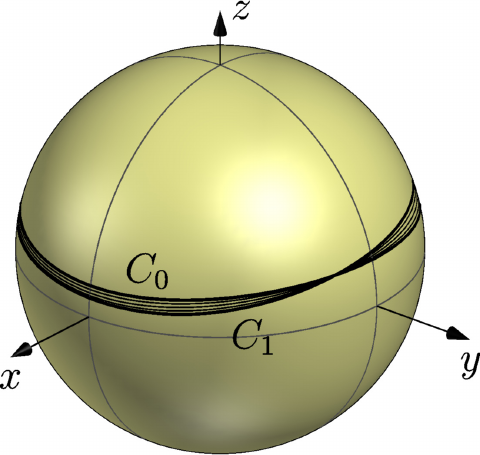}
    \caption{$C_0$, $C_1$ and $C_\lambda$ for $\lambda=0.2, 0.5,
      0.8$}
    \label{fig:sph_CounterExample}
  \end{minipage}%
  \begin{minipage}[b]{0.5\linewidth}
    \rule{\linewidth}{0pt}
    \centering
    \includegraphics{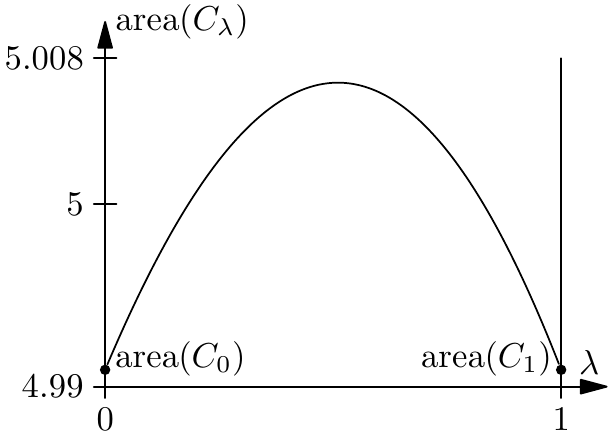}
    \caption{Area of the in-between conics}
    \label{fig:sph_CounterExampleSize}
  \end{minipage}
\end{figure}

The behaviour illustrated in Example~\ref{ex:1} is in contrast to the
situation in the Euclidean plane, where a convexity property of the
size function similar to Lemma~\ref{lem:1} guarantees that the
in-between ellipsoids $C_\lambda$ can be translated so that they are
completely contained in $C_0$ or $C_1$ (the ``Translation Lemma~6'' of
\cite{schroecker08:_uniqueness_results_ellipsoids}). In particular,
the size of every in-between ellipsoid is strictly smaller than the
size of $C_0$ and~$C_1$ (for \emph{any} reasonable size function). An
important observation is that the conics $C_0$ and $C_1$ in
Example~\ref{ex:1} are rather large and ``far away'' from a pair of
Euclidean conics. Thus, one might hope that uniqueness via in-between
conics can be shown for sufficiently small conics. This is indeed the
case. A precise formulation is given below. It requires an auxiliary
result:

\begin{lemma}
  \label{lem:unique-zero}
  The function
  \begin{equation}
    \label{eq:24}
    J(v) = \int_0^w \frac{1+v-3t^2}{\sqrt{1-t^2}} \dif t +
           \int_w^1 \frac{(1+v-3t^2)\sqrt{1+v}}{\sqrt{1-t^2}\sqrt{1+v-t^2}} \dif t
    \quad\text{with}\quad
    w = \sqrt{\frac{1+v}{3}}
  \end{equation}
  is strictly monotone increasing in $[0,2]$. In particular, it has
  precisely one zero $v_0 \approx 0.685935$ in this interval.
\end{lemma}

\begin{proof}
  Denote the first integral in \eqref{eq:24} by $J_1(v)$ and the
  second by $J_2(v)$. Clearly, $J_1(v)$ is strictly monotone
  increasing in $v$ because the integrand and the upper integration
  bound are strictly monotone increasing. The integrand of $J_2(v)$ is
  negative. Thus, increasing the lower integration bound $w$ will also
  increase the integral. Moreover, the second integrand is strictly
  monotone increasing in $v$ as well. In order to see this, we compute
  its first derivative with respect to $v$:
  \begin{equation}
    \label{eq:25}
    \frac{2v^2+(4-3t^2)v+3t^4-3t^2+2}
         {2\sqrt{1-t^2}\sqrt{1+v}\;(1+v-t^2)^{3/2}}.
  \end{equation}
  The denominator is positive, the numerator is strictly monotone
  increasing in $v$ and, for $v = 0$, attains the positive value
  $3t^4-3t^2+2$. Thus, the derivative is positive. This implies that
  the integrand of $J_2$ is strictly monotone increasing and the same
  is true for the function $J$ defined in~\eqref{eq:24}.
\end{proof}

\begin{theorem}
  \label{th:3}
  Denote the unique zero of $\eqref{eq:24}$ in $[0,2]$ by $v_0$ and
  let $\Rho = \arctan(v_0^{-1/2})$. The enclosing conic of minimal
  area of a compact subset $F$ of the elliptic plane is unique if the
  following two conditions are met:
  \begin{enumerate}
  \item The elliptic convex hull of $F$ contains a circle of radius
    $\varrho > 0$. In particular, $F$ is full-dimensional.
  \item There exists an enclosing conic of $F$ whose area is less than
    $\area(\Rho, \varrho)$, computed by means of
    Equation~\eqref{eq:7}.
  \end{enumerate}
\end{theorem}

Note that the requirements of this theorem imply restrictions on the
set $F$ and all candidates for enclosing conics of minimal
area:
\begin{itemize}
\item The diameter of $F$ is less than $\Rho$, that is, $F$ is bounded
  by a fixed value derived from the function $J(v)$.
\item Any enclosing conic has a minor semi-axis lengths $\beta \ge
  \varrho$ and any minimal enclosing conic has a major semi-axis
  length $\alpha < \Rho$. In other words, if we insert the value $v =
  1/\tan^2\alpha$ into the function $J(v)$, the result is positive.
\end{itemize}

The basic idea of the proof is not different from the proofs of
Theorems~\ref{th:1} and \ref{th:2} but the details are more
involved. Existence of the minimal enclosing conic follows from the
usual compactness argument. In order to prove uniqueness, we assume
existence of two conics $C_0$ and $C_1$ of minimal area that contain
$F$. Note that their minor semi-axis length is not smaller than
$\varrho$ and their major semi-axis lengths is smaller than $\Rho$. We
show existence of an in-between conic $C_\lambda$ such that
$\area(C_\lambda) < \area(C_0) = \area(C_1)$. This we do by proving
that $\area(C_\lambda)$ is strictly monotone decreasing in the
vicinity of $\lambda = 0$ (possibly after interchanging $C_0$ and
$C_1$). Thus, we obtain a contradiction to the assumed minimality of
$C_0$ and $C_1$. In order to show that $\area(C_\lambda)$ is strictly
monotone decreasing in the vicinity of $\lambda = 0$, we compare the
derivative of the area function with respect to $\lambda$ to the
derivative of the area function in a suitably constructed case with
concentric conics (Lemma~\ref{lem:half-turn}). The details of this
proof span until the end of this section. Auxiliary results of
technical nature are proved in the appendix.

\subsubsection{Assumptions on the semi-axis lengths}
\label{sec:assumptions-semi-axis-lengths}

For $i = 0$, $1$ we denote the matrix describing the conic $C_i$ by
$M_i$. Its eigenvalues are $\nu_{i,1} \ge \nu_{i,2} > 0$ and
$\nu_{i,3} = -1$. It is no loss of generality to make a few
assumptions on these values:

If $\nu_{0,1} = \nu_{0,2}$ and $\nu_{1,1} = \nu_{1,2}$, equality of
areas of $C_0$ and $C_1$ implies $\nu_{0,1} = \nu_{0,2} = \nu_{1,1} =
\nu_{1,2}$. In this case both conics are congruent circles with two
real intersection points $s_1$, $s_2$. There exists a circle with
$s_1$ and $s_2$ as end-points of a diameter which also contains the
common interior of $C_0$ and $C_1$. It is smaller than $C_0$ and $C_1$
and thus contradicts the assumed minimality of these conics.
Henceforth, we exclude equality of all four positive eigenvalues of
$M_0$ and~$M_1$. Then equality of areas of $C_0$ and $C_1$ implies
that these eigenvalues can be nested (possibly after interchanging
$C_0$ and $C_1$) according to
\begin{equation}
  \label{eq:26}
  \nu_{0,1} > \nu_{1,1} \ge \nu_{1,2} > \nu_{0,2}.
\end{equation}

\subsubsection{Derivative of the area function}
\label{sec:derivtive-area-function}

As already mentioned, a contradiction to the assumed minimality of
$C_0$ and $C_1$ arises if we can show that
\begin{equation}
  \label{eq:27}
  \od{\area(C_\lambda)}{\lambda} \Big| _{\lambda=0} < 0.
\end{equation}
The advantage of this ``local'' approach is that the derivative
\eqref{eq:27} can be computed from the derivatives of the eigenvalues
$\nu_{i,1}$ and $\nu_{i,2}$ with respect to $\lambda$ and that these
derivatives do not require explicit expressions of the eigenvalues as
functions of~$\lambda$.

We assume that $C_0$ is of the normal form \eqref{eq:2} and $C_1$ is
obtained from a conic in this normal form by a rotation about an axis
through the center of $S^2$, that is,
\begin{equation}
  \label{eq:28}
  M_0 =
  \begin{pmatrix}
    \nu_{0,1} & 0         & 0 \\
    0         & \nu_{0,2} & 0 \\
    0         & 0         & -1
  \end{pmatrix}, \quad M_1 = Q \cdot
  \begin{pmatrix}
    \nu_{1,1} & 0         & 0 \\
    0         & \nu_{1,2} & 0 \\
    0         & 0         & -1
  \end{pmatrix} \cdot Q^{-1},
\end{equation}
with the rotation matrix
\begin{equation}
  \label{eq:29}
  Q =
  \begin{pmatrix}
    q_0^2 + q_1^2 - q_2^2 - q_3^2 & 2(q_1 q_2 - q_0 q_3)          & 2(q_1 q_3 + q_0 q_2) \\
    2(q_1 q_2 + q_0 q_3)          & q_0^2 - q_1^2 + q_2^2 - q_3^2 & 2(q_2 q_3 - q_0 q_1) \\
    2(q_1 q_3 - q_0 q_2)          & 2(q_2 q_3 + q_0 q_1)          & q_0^2 - q_1^2 - q_2^2 + q_3^2
  \end{pmatrix}
\end{equation}
and $q_0^2 + q_1^2 + q_2^2 + q_3^2 = 1$. The rotation angle $\theta$
is given by $q_0 = \cos 2\theta$, the axis direction is
$(q_1,q_2,q_3)^\tp$.  The matrix $M_\lambda$ of the in-between conic
is computed according to \eqref{eq:16}. Its ordered eigenvalues
$\nu_1(\lambda) \ge \nu_2(\lambda) \ge \nu_3(\lambda)$ are functions
of $\lambda$ and in the vicinity of $\lambda = 0$ we have
$\nu_1(\lambda) > \nu_2(\lambda) > 0 > \nu_3(\lambda)$. The
eigenvalues are implicitly defined as roots of the characteristic
polynomial $P(\lambda, \nu(\lambda)) = \det(M_\lambda - \nu I_3)$ of
$M_\lambda$, $I_3$ being the three by three identity matrix. We know
the values of these roots for $\lambda=0$:
\begin{equation}
  \label{eq:30}
  \nu_1(0) = \nu_{0,1}, \quad \nu_2(0) = \nu_{0,2}, \quad \nu_3(0) = -1.
\end{equation}
By implicit derivation we have
\begin{equation}
  \label{eq:31}
  \dod{\nu_i}{\lambda}(0) =
  -\frac{\pd{P}{\lambda}(0, \nu_i(0))}{\pd{P}{\nu}(0, \nu_i(0))},
  \quad i = 1,2,3,
\end{equation}
which gives us the derivatives of the three eigenvalues of $M_\lambda$
at $\lambda = 0$. Furthermore, we can compute
\begin{equation}
  \label{eq:32}
  \dpd{\area(\nu_{0,1}, \nu_{0,2},
    \nu_{0,3})}{\nu_{0,i}}, \quad i = 1,2,3
\end{equation}
and, using the chain rule
\begin{equation}
  \label{eq:33}
  \dpd{\area(C_\lambda)}{\lambda}\Big|_{\lambda=0}
  = \dpd{\area}{\nu_{0,1}} \Part{\nu_{\lambda,1}}{\lambda}\Big|_{\lambda=0}
  + \dpd{\area}{\nu_{0,2}} \Part{\nu_{\lambda,2}}{\lambda}\Big|_{\lambda=0}
  + \dpd{\area}{\nu_{0,3}} \Part{\nu_{\lambda,3}}{\lambda}\Big|_{\lambda=0},
\end{equation}
we find
\begin{multline}
  \label{eq:34}
  \dpd{\area(C_\lambda)}{\lambda}\Big|_{\lambda=0} = \\
  -\frac{1}{2} \int_{-\pi}^\pi \frac{1}{N} \bigl( \sin^2\varphi (
  q_{1,1}^2 \nu_{1,1} + q_{1,2}^2 \nu_{1,2} - q_{1,3}^2 + \nu_{0,1} (
  q_{3,1}^2 \nu_{1,1} + q_{3,2}^2 \nu_{1,2} - q_{3,3}^2) )\\
  + \cos^2\varphi ( q_{2,1}^2 \nu_{1,1} + q_{2,2}^2 \nu_{1,2} -
  q_{2,3}^2 + \nu_{0,2} ( q_{3,1}^2 \nu_{1,1} + q_{3,2}^2 \nu_{1,2} -
  q_{3,3}^2) ) \bigr) \dif\varphi,
\end{multline}
where
\begin{equation}
  \label{eq:35}
  N = ( \nu_{0,1} \sin^2\varphi + \nu_{0,2} \cos^2\varphi )^{1/2}
      (1 + \nu_{0,1} \sin^2\varphi + \nu_{0,2} \cos^2\varphi )^{3/2}
\end{equation}
and $q_{i,j}$ are the entries of the rotation matrix
\eqref{eq:29}. Equation~\eqref{eq:34} expresses the derivative of the
area function with respect to $\lambda$ at $\lambda = 0$ in terms of
the two initially given conics $C_0$ and $C_1$.  It will be convenient
to write \eqref{eq:34} in terms of the first and second complete
elliptic integrals
\begin{equation}
  \label{eq:36}
  K(z) = \int_0^1 \frac{1}{\sqrt{1-t^2} \sqrt{1-z^2t^2}} \dif t
  \quad\text{and}\quad
  E(z) = \int_0^1 \frac{\sqrt{1-z^2t^2}}{\sqrt{1-t^2}} \dif t.
\end{equation}
Since we will evaluate them only at
\begin{equation}
  \label{eq:37}
  f = \sqrt{\frac{\nu_{0,1}-\nu_{0,2}}{\nu_{0,1}(1+\nu_{0,2})}},
\end{equation}
we use the abbreviations $\ellE := E(f)$ and $\ellK := K(f)$. By
\eqref{eq:26}, $f$ is always real.

Substituting $x = \nu_{0,1} \sin^2\varphi + \nu_{0,2} \cos^2\varphi$
into \eqref{eq:34} and noting that $(x-\nu_{0,1})(\nu_{0,2}-x) =
(\nu_{0,1} - \nu_{0,2})^2\cos^2\varphi \sin^2\varphi$ we can express
the derivative of the area function in terms of
elliptic integrals:
\begin{multline}
  \label{eq:38}
    \dpd{\area(C_\lambda)}{\lambda}\Big|_{\lambda=0} =
    \frac{2}{\sqrt{\nu_{0,1} (1 + \nu_{0,2})} (\nu_{0,1} - \nu_{0,2})
      (1 + \nu_{0,1})} \\
    \bigl( (1 + \nu_{0,1}) ( -\nu_{0,1} (q_{2,1}^2 \nu_{1,1} +
    q_{2,2}^2 \nu_{1,2} - q_{2,3}^2 ) + \nu_{0,2} ( q_{1,1}^2
    \nu_{1,1} + q_{1,2}^2 \nu_{1,2} - q_{1,3}^2) ) \ellK \\
    - \nu_{0,1} ( \nu_{0,1} (q_{3,1}^2 \nu_{1,1} + q_{3,2}^2
    \nu_{1,2} - q_{3,3}^2 - q_{2,1}^2 \nu_{1,1} - q_{2,2}^2 \nu_{1,2}
    + q_{2,3}^2) \\
    + \nu_{0,2} (- q_{3,1}^2 \nu_{1,1} - q_{3,2}^2 \nu_{1,2} +
    q_{3,3}^2 + q_{1,1}^2 \nu_{1,1} + q_{1,2}^2 \nu_{1,2} - q_{1,3}^2)
    \\
    + q_{1,1}^2 \nu_{1,1} + q_{1,2}^2 \nu_{1,2} - q_{1,3}^2 -
    q_{2,1}^2 \nu_{1,1} - q_{2,2}^2 \nu_{1,2} + q_{2,3}^2 ) \ellE
    \bigr).
\end{multline}

\subsubsection{The half-turn lemma}
\label{sec:half-turn-lemma}

From the proof of Theorem~\ref{th:2} we already know that
\eqref{eq:38} is negative if $C_0$ and $C_1$ are concentric. We show
negativity in the general case by comparison with a concentric
situation. This is a direct consequence of the ``Half-Turn Lemma''
below. The basic idea already occurred in
\cite{schroecker08:_uniqueness_results_ellipsoids} in form of a
``Translation Lemma''.

\begin{lemma}[Half-Turn Lemma]
  \label{lem:half-turn}
  Consider three conics $C_0$, $D_1$, $D_2$ of equal area and with
  major semi-axis lengths smaller than $\Rho$ as defined in
  Theorem~\ref{th:3}. Assume that
  \begin{itemize}
  \item $C_0$ and $D_1$ are concentric, and
  \item $D_2$ is obtained from $D_1$ by a half-turn, that is, a
    rotation through an angle of~$\pi$.
  \end{itemize}
  Then the area of $C_{\lambda,1} = (1-\lambda) C_0 + \lambda D_1$ is
  smaller than the area of $C_{\lambda,2} = (1-\lambda) C_0 + \lambda
  D_2$, at least in the vicinity of $\lambda = 0$.
\end{lemma}

In order to prove Lemma~\ref{lem:half-turn} it is sufficient to
compare the derivatives of the area of $C_{\lambda,1}$ and
$C_{\lambda,2}$ with respect to $\lambda$ at $\lambda = 0$. Computing
these derivatives is easily accomplished by substituting appropriate
values for the entries of the matrix \eqref{eq:29} into \eqref{eq:38}:

The conic $D_1$ can be obtained from a conic in normal form
\eqref{eq:2} by a rotation about $(0, 0, 1)^\tp$ through an angle
$\zeta$. We can compute the corresponding matrix $M_1$ as in
\eqref{eq:28} with suitable entries chosen for the matrix $Q$ in
\eqref{eq:29}. Substituting these values into Equation~\eqref{eq:38}
yields
\begin{multline}
  \label{eq:39}
  \dod{\area(C_{\lambda, 1})}{\lambda} \Big|_{\lambda=0} =\\
  \frac{1}{N_1} \big( (A \sin^2\zeta + C \cos^2\zeta)
  \nu_{1,1} + (A \cos^2\zeta + C \sin^2\zeta) \nu_{1,2} -
  B \big),
\end{multline}
where
\begin{equation}
  \label{eq:40}
  \begin{gathered}
    N_1 = \sqrt{\nu_{0,1} (1 + \nu_{0,2})} (\nu_{0,1} - \nu_{0,2}) (1 + \nu_{0,1}),\quad
    A = 2 \nu_{0,1} (1 + \nu_{0,1}) (\ellE -\ellK),\\
    B = -2 \nu_{0,1} (\nu_{0,1} - \nu_{0,2}) \ellE,\quad
    C = 2 (1 + \nu_{0,1}) \nu_{0,2} \ellK - 2 \nu_{0,1} (1 + \nu_{0,2}) \ellE.
  \end{gathered}
\end{equation}

The conic $D_2$ is obtained by a half-turn from $D_1$ about the
rotation axis defined by the unit vector $r = (r_1, r_2,
r_3)^\tp$. The point $r$ can be chosen as one of the two mid-points of
the centers of $C_0$ and $C_2$. It is no loss of generality to assume
$r_1^2 + r_2^2 - r_3^2 \le 0$ since otherwise we take the second
mid-point. In particular, we can always assume
\begin{equation}
  \label{eq:41}
  r_3^2-r_1^2 \ge 0.
\end{equation}
The matrix $Q$ in \eqref{eq:29} is the product of the rotation matrix
about $(0, 0, 1)^\tp$ through $\zeta$ and a half-turn rotation matrix
about the unit vector $r$. Substituting according entries into
Equation~\eqref{eq:34} yields
\begin{equation}
  \label{eq:42}
  \begin{gathered}
  \dod{\area(C_{\lambda, 2})}{\lambda} \Big|_{\lambda=0} = \\
  \frac{1}{N_1} \Big( \big( (A \sin^2\zeta + C
  \cos^2\zeta) \nu_{1,1} + (A \cos^2\zeta + C
  \sin^2\zeta) \nu_{1,2} - B \big) (r_1^4 + r_2^4 + r_3^4)\\
  + \big( ( - 2 C \cos^2\zeta + 4 C \sin^2\zeta - 2 A
  \sin^2\zeta + 4 A \cos^2\zeta) \nu_{1,1}\\
  + ( -2C \sin^2\zeta + 4 C \cos^2\zeta - 2 A
  \cos^2\zeta + 4 A \sin^2\zeta) \nu_{1,2} - 2 B \big)
  r_1^2 r_2^2\\
  + \big( (4 B \cos^2\zeta - 2 C \cos^2\zeta + 2 A
  \sin^2\zeta) \nu_{1,1}\\
  + (4 B \sin^2\zeta - 2 C \sin^2\zeta + 2 A
  \cos^2\zeta) \nu_{1,2} + 2 B - 4 C \big) r_1^2 r_3^2\\
  + \big( ( - 2 A \sin^2\zeta + 4 B \sin^2\zeta + 2 C
  \cos^2\zeta) \nu_{1,1}\\
  + ( -2A \cos^2\zeta + 4 B \cos^2\zeta + 2 C
  \sin^2\zeta) \nu_{1,2} - 4 A + 2 B \big) r_2^2 r_3^2\\
  + \big( 4 \cos\zeta \sin\zeta ( - C + A) \nu_{1,1} - 4
  \cos\zeta \sin\zeta ( - C + A) \nu_{1,2} \big) (r_1
  r_2^3 - r_1^3 r_2)\\
  + \big( - 4 \cos\zeta \sin\zeta ( - 2 B + C + A) \nu_{1,1} + 4
  \cos\zeta \sin\zeta ( - 2 B + C + A) \nu_{1,2} \big) r_1
  r_2 r_3^2 \Big)
  \end{gathered}
\end{equation}
with $N_1$, $A$, $B$, and $C$ as in~\eqref{eq:40}.

Now we are going to prove the inequality
\begin{equation}
  \label{eq:43}
  \dod{\area(C_{\lambda, 1})}{\lambda} \Big|_{\lambda=0} \ge
  \dod{\area(C_{\lambda, 2})}{\lambda} \Big|_{\lambda=0}
\end{equation}
under the additional assumption
\begin{equation}
  \label{eq:44}
  B < A < C < 0.
\end{equation}
This indeed holds true: $A < 0$ follows from $\ellE < \ellK$, $B < 0$
is a consequence of \eqref{eq:26}, $C < 0$ will be shown in
Lemma~\ref{lem:3} and $B < A < C$ in Lemma~\ref{lem:4} and
Lemma~\ref{lem:5}.

We substitute $\zeta = 2 \arctan t$, that is,
$\cos\zeta=(1-t^2)/(1+t^2)$, $\sin\zeta = 2t/(1+t^2)$, in
\begin{equation}
  \label{eq:45}
  \dod{\area(C_{\lambda, 2})}{\lambda} \Big|_{\lambda=0} -
  \dod{\area(C_{\lambda, 1})}{\lambda} \Big|_{\lambda=0}
\end{equation}
to obtain a rational expression in $t$. Clearing the (positive)
denominator, we arrive at a polynomial $P(t)$ of degree four in
$t$. We have to show that $P(t)$ is strictly negative on $(0,1)$. For
that purpose we use a typical technique and write
\begin{equation}
  \label{eq:46}
  P(t) = \sum_{i=0}^4 B^4_i(t) p_i
\end{equation}
where $B^4_i(t) = \binom{4}{i} (1-t)^{4-i}t^i$ is the $i$-th Bernstein
polynomial of degree four. Because of $B^4_i(t) \in [0,1]$ and
$\sum_{i=0}^4 B^4_i(t) \equiv 1$, $P(t)$ is a convex combination of
the Bernstein coefficients $p_0$, \ldots, $p_4$. Hence, the polynomial
$P(t)$ is certainly negative on $(0,1)$ if we can show that no
Bernstein coefficient $p_i$ is positive and at least one is
negative. The Bernstein coefficients are:
\begin{equation}
  \label{eq:47}
  \begin{gathered}
    p_0 =
    (C \nu_{1,1} + A \nu_{1,2} - B) \big( (r_1^2 + r_2^2 + r_3^2)^2 - 1 \big)
    + 4 r_1^2 r_2^2 (A - C) (\nu_{1,1} - \nu_{1,2}) \\
    + 4 r_1^2 r_3^2 (B - C) (1 + \nu_{1,1})
    + 4 r_2^2 r_3^2 (B - A) (1 + \nu_{1,2}).
  \end{gathered}
\end{equation}
Because of $r_1^2 + r_2^2 + r_3^2 = 1$, $B < A < C$, and $\nu_{1,1}
\ge \nu_{1,2}$ it follows that $p_0$ is not positive. It equals zero
if and only if $r_1 = r_2 = 0$ which is only possible in the
concentric case and therefore can be excluded. Thus, $p_0$ is
negative.
\begin{equation}
  \label{eq:48}
  \begin{gathered}
    p_1 = (C \nu_{1,1} + A \nu_{1,2} - B) \big( (r_1^2 + r_2^2 + r_3^2)^2 - 1 \big)
    + 4 r_1^2 r_2^2 (A - C) (\nu_{1,1} - \nu_{1,2}) \\
    + 4 r_1^2 r_3^2 (B - C) (1 + \nu_{1,1})
    + 4 r_2^2 r_3^2 (B - A) (1 + \nu_{1,2})\\
    + 2 r_1 r_2 r_3^2 (2B - A - C) (\nu_{1,1} - \nu_{1,2})
    + 2(r_1 r_2^3 - r_1^3 r_2) (A - C) (\nu_{1,1} - \nu_{1,2}).
  \end{gathered}
\end{equation}
Because of $r_1^2 + r_2^2 + r_3^2 = 1$, $B < A < C$, and $\nu_{1,1}
\ge \nu_{1,2}$ neither the first nor the second row of~\eqref{eq:48}
is positive. In Lemma~\ref{lem:6} we show that the third row is
not positive either. Thus, $p_1 \le 0$, as required.
\begin{equation}
  \label{eq:49}
  \begin{gathered}
    3p_2 = 2((A + C) (\nu_{1,1} + \nu_{1,2}) - 2B) \big(
    (r_1^2 + r_2^2 + r_3^2)^2 - 1 \big)\\
    + 8 r_1^2 r_3^2 (B - C) (2 + \nu_{1,1} + \nu_{1,2})
    + 8 r_2^2 r_3^2 (B - A) (2 + \nu_{1,1} + \nu_{1,2})\\
    + 12 r_1 r_2 r_3^2 (2B - A - C) (\nu_{1,1} - \nu_{1,2})
    + 12 (r_1 r_2^3 - r_1^3 r_2) (A - C) (\nu_{1,1} - \nu_{1,2}).
  \end{gathered}
\end{equation}
The non-positivity of $p_2$ is shown similar to that of~$p_1$.
\begin{equation}
  \label{eq:50}
  \begin{gathered}
    p_3 = 2 (A \nu_{1,1} + C \nu_{1,2} - B) \big( (r_1^2 + r_2^2 +
    r_3^2)^2 - 1 \big)
    + 8 r_1^2 r_2^2 (C - A) (\nu_{1,1} - \nu_{1,2})\\
    + 8 r_1^2 r_3^2 (B - C) (1 + \nu_{1,2})
    + 8 r_2^2 r_3^2 (B - A) (1 + \nu_{1,1})\\
    + 4 r_1 r_2 r_3^2 (2B - A - C) (\nu_{1,1} - \nu_{1,2}) + 4 (r_1
    r_2^3 - r_1^3 r_2) (A - C) (\nu_{1,1} - \nu_{1,2}).
  \end{gathered}
\end{equation}
The non-positivity of the last row is shown in Lemma~\ref{lem:6}. The
non-positivity of the first and second row is shown in
Lemma~\ref{lem:7}.
\begin{equation}
  \label{eq:51}
  \begin{gathered}
    p_4 =
    4 (A \nu_{1,1} + C \nu_{1,2} - B) \big( (r_1^2 + r_2^2 + r_3^2)^2 - 1 \big)\\
    + 16 r_1^2 r_2^2 (C - A) (\nu_{1,1} - \nu_{1,2})
    + 16 r_1^2 r_3^2 (B - C) (1 + \nu_{1,2})\\
    + 16 r_2^2 r_3^2 (B - A) (1 + \nu_{1,1}).
  \end{gathered}
\end{equation}
We show the non-positivity of the second and third row in
Lemma~\ref{lem:7}.

Summarizing we can state that no Bernstein coefficient is positive and
at least $p_0$ is negative. Therefore we conclude that
Equation~\eqref{eq:45} is valid and the Half-Turn Lemma holds
true. This also concludes the proof of Theorem~\ref{th:3}.

\section{Minimal enclosing conics of line-sets}
\label{sec:line-sets}

As a final result, we would like to present the elliptic counter-part
of the uniqueness theorem of \cite{schroecker07:_minim_hyper} for
minimal enclosing hyperbolas in the Euclidean plane. The perfect
duality between points and lines in the elliptic plane allows to
derive this result without additional work as simple corollary to
Theorems~\ref{th:1}, \ref{th:2}, and~\ref{th:3}.

The definition of a measure for a set of lines $\lSet$ that is
invariant with respect to elliptic transformations is
straightforward. The key-ingredient is the absolute polarity. We
describe it for the bundle model of the elliptic plane. Here, a
straight line $L$ is a plane through the center of $S^2$. The plane
normal defines a point $l$ which is called the \emph{absolute pole} of
$L$.  Conversely, $L$ is the \emph{absolute polar} of $l$. We write $l
= p(L)$ and $L = p(l)$.  The measure $m(\lSet)$ of a line-set $\lSet$
is defined as the area of the absolute poles of lines in $\lSet$:
\begin{equation}
  \label{eq:52}
  m(\lSet) := \area \{p(L) \mid L \in \lSet\}.
\end{equation}

A conic $C$ is said to contain a set $\lSet$ of straight lines if
every member $L \in \lSet$ has at most one real intersection point
with $C$. In this sense, one may ask for the enclosing conic of
minimal size with respect to the measure $m$
(Figure~\ref{fig:spherical-hyperbola}). As corollary to
Theorem~\ref{th:3} we can state

\begin{corollary}
  Let $\lSet \subset S^2$ be a set of lines such that its polar set
  $p(\lSet)$ satisfies the requirements of Theorem~\ref{th:1}. Among
  all conics with two given axes that contain $\lSet$ there exists
  exactly one of minimal measure.
\end{corollary}

\begin{corollary}
  Let $\lSet \subset S^2$ be a set of lines such that its polar set
  $p(\lSet)$ satisfies the requirements of Theorem~\ref{th:2}. Among
  all conics with given center that contain $\lSet$ there exists
  exactly one of minimal measure.
\end{corollary}

\begin{corollary}
  \label{cor:3}
  Let $\lSet \subset S^2$ be a compact set of lines that is contained
  in a circle of radius $\varrho^{-1} > 0$. Assume there exists an
  enclosing conic of $\lSet$ whose measure is larger than the measure
  for the set of lines contained in a conic of semi-axes lengths
  $\Rho^{-1}$ and $\varrho^{-1}$. Then the enclosing conic of minimal
  measure $m$ to the line set $\lSet$ is unique.
\end{corollary}

\begin{figure}
  \begin{minipage}[b]{0.5\linewidth}
    \rule{\linewidth}{0pt}
    \centering
    \includegraphics{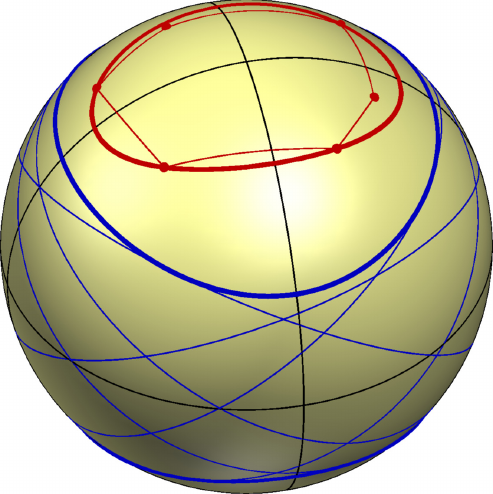}
  \end{minipage}%
  \begin{minipage}[b]{0.3\linewidth}
    \rule{\linewidth}{0pt}
    \centering
    \includegraphics{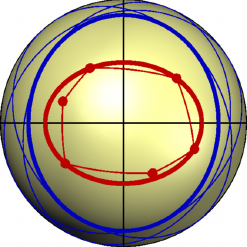}\\[1.5ex]
    \includegraphics{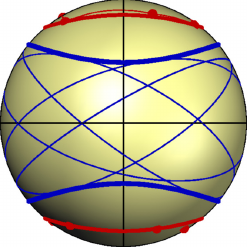}
  \end{minipage}
  \caption{Minimal enclosing conics of a point set and a line set in
    three orthographic projections}
  \label{fig:spherical-hyperbola}
\end{figure}

\section{Conclusion}
\label{sec:conclusion}

We have proved uniqueness of the enclosing conic of minimal area in
the elliptic plane. If the minimizer is sought within a set of conics
with prescribed axes or center, it is unique at any rate
(Theorem~\ref{th:1} and Theorem~\ref{th:2}). In the most general
setting we can show uniqueness only under additional assumptions on
the enclosed set $F$ (Theorems~\ref{th:3}). In particular, $F$ must be
contained in a circle of radius $\Rho$. Given the fact that the
diameter of the elliptic plane is $\pi$, the bounds on the size of $F$
seem acceptable for many applications. The question whether the
conditions on $F$ can be relaxed remains open. Example \ref{ex:1} on
Page \pageref{ex:1} shows that the method of in-between conics is not
capable of proving uniqueness without additional constraints on $F$.

Open questions in the area of extremal quadrics are numerous. Starting
from this article it would be natural to consider uniqueness results
for size functions different from the area and uniqueness of maximal
inscribed conics, generalizations to higher dimensions and other
non-Euclidean geometries, for example the hyperbolic plane.

\appendix

\section{Proofs of auxiliary results}

In this appendix we prove technical results which are needed in the
proofs of our main theorems but are probably of little interest
otherwise.

\begin{lemma}
  \label{lem:3}
  For $C$ as defined in \eqref{eq:40} we have $C < 0$.
\end{lemma}

\begin{proof}
  Inserting \eqref{eq:36} and \eqref{eq:37} into the defining
  Equation~\eqref{eq:40} of $C$ we obtain the integral representation
  \begin{equation}
    \label{eq:53}
    C = -2 \int_0^1
    \frac{(\nu_{0,1}-\nu_{0,2})\sqrt{1-t^2}\sqrt{\nu_{0,1}(1+\nu_{0,2})}}
         {\sqrt{t^2(\nu_{0,2}-\nu_{0,1})+\nu_{0,1}(1+\nu_{0,2})}}
    \dif t
  \end{equation}
  Obviously, the integral is positive so that $C$ is negative.
\end{proof}

\begin{lemma}
  \label{lem:4}
  For $A$ as defined in~\eqref{eq:40} we have $A < C$.
\end{lemma}

\begin{proof}
  We view $A$ and $C$ as functions of $\nu_{0,1}$ and $\nu_{0,2}$ and
  show that the inequality $A < C$ holds for every
  $\nu_{0,1}$-parameter line. For $\nu_{0,1} = \nu_{0,2}$ (which we
  generally exclude, see
  Subsection~\ref{sec:assumptions-semi-axis-lengths}) we have $\ellE =
  \ellK$. Therefore, the function $D(\nu_{0,1},\nu_{0,2}) =
  A(\nu_{0,1},\nu_{0,2}) - C(\nu_{0,1},\nu_{0,2})$ vanishes for
  $\nu_{0,1} = \nu_{0,2}$. The same is true for its derivative with
  respect to $\nu_{0,1}$ which can be computed as
  \begin{equation}
    \label{eq:54}
    \dpd{D}{\nu_{0,1}}(\nu_{0,1},\nu_{0,2}) =
    \frac{2\nu_{0,1}(2\nu_{0,1}+\nu_{0,2})(\ellE - \ellK)
         -(2\nu_{0,1}+\nu_{0,2})\ellK + 3\nu_{0,1}\ellE}
         {\nu_{0,1}}.
  \end{equation}
  If we can show that $D$ is strictly concave in $\nu_{0,1}$ on
  $\nu_{0,1} > \nu_{0,2}$ we may also conclude $D < 0$ for $\nu_{0,1}
  > \nu_{0,2}$. The second partial derivative with respect to
  $\nu_{0,1}$ reads
  \begin{multline}
    \label{eq:55}
    \dpd[2]{D}{\nu_{0,1}} = \frac{1}{2 \nu_{0,1}^2 (1+\nu_{0,1})}
    \big( ( 8\nu_{0,1}^3 + 4\nu_{0,1}^2 -2\nu_{0,1} - \nu_{0,1} \nu_{0,2} ) \ellE\\
    - ( 8\nu_{0,1}^3 + 8\nu_{0,1}^2 - \nu_{0,1} \nu_{0,2} - \nu_{0,2} )\ellK \big).
  \end{multline}
  We have to show that it is negative for $\nu_{0,1} > \nu_{0,2}$.
  Because of $\ellE < \ellK$ we have
  \begin{equation}
    \label{eq:56}
    (8\nu_{0,1}^3 + 4\nu_{0,1}^2)(\ellE - \ellK) < 0.
  \end{equation}
  By subtracting \eqref{eq:56} from \eqref{eq:55} we see that
  \begin{equation}
    \label{eq:57}
    (-2\nu_{0,1} - \nu_{0,1} \nu_{0,2}) \ellE -
    (4\nu_{0,1}^2 - \nu_{0,1} \nu_{0,2} - \nu_{0,2}) \ellK < 0
  \end{equation}
  is sufficient for the negativity of \eqref{eq:55}. Moreover, we have
  $\nu_{0,1}\nu_{0,2} < \nu_{0,1}^2$ so that it enough to show
  \begin{equation}
    \label{eq:58}
    (-2\nu_{0,1} - \nu_{0,1} \nu_{0,2}) \ellE -
    (3\nu_{0,1}^2 - \nu_{0,2}) \ellK < 0.
  \end{equation}
  Clearly the inequality
  \begin{equation}
    \label{eq:59}
    (-\nu_{0,1}-\nu_{0,1}\nu_{0,2})\ellE - 3\nu_{0,1}^2\ellK < 0
  \end{equation}
  holds true. Subtracting \eqref{eq:59} from \eqref{eq:58}, we arrive
  at
  \begin{equation}
    \label{eq:60}
    \nu_{0,2}\ellK -\nu_{0,1}\ellE < 0.
  \end{equation}
  We write $\nu_{0,2}\ellK - \nu_{0,1} \ellE$ in its integral form
  \begin{equation}
    \label{eq:61}
    \nu_{0,2}\ellK - \nu_{0,1} \ellE = -\frac{\nu_{0,1} - \nu_{0,2}}{1 +
      \nu_{0,2}} \int_0^1 \frac{1 + \nu_{0,2} - t^2}{\sqrt{1-t^2}
      \sqrt{1-f^2 t^2}} \dif t
  \end{equation}
  and observe that the factor before the integral is negative while
  the denominator of the integrand is positive. Moreover, the
  numerator of the integrand is linear in $t^2$. For $t = 0$ it equals
  $1+\nu_{0,2} > 0$ and for $t = 1$ it equals $\nu_{0,2} > 0$. Thus,
  it is positive for $t \in [0,1]$. We conclude that \eqref{eq:57}
  holds true. Hence, the $\nu_{0,1}$-parameter lines of $D$ are
  strictly concave so that indeed $A < C$.
\end{proof}

\begin{lemma}
  \label{lem:5}
  If $\nu_{0,2} > v_0$ (this is implied by the assumptions of
  Theorem~\ref{th:3}) we have $B < A$.
\end{lemma}

\begin{proof}
  We show that $2A - B - C \ge 0$. Because of $A - C < 0$ this implies
  $B < A$. We have
  \begin{equation}
    \label{eq:62}
    2A - B - C = 6 \nu_{0,1} (1 + \nu_{0,1}) \ellE - 2 (2 \nu_{0,1} +
    \nu_{0,2}) (1 + \nu_{0,1}) \ellK.
  \end{equation}
  This expression is not negative if and only if
  \begin{equation}
    \label{eq:63}
    3 \nu_{0,1} \ellE - (2 \nu_{0,1} + \nu_{0,2}) \ellK
    \ge 0.
  \end{equation}
  Writing $\ellE$ and $\ellK$ in their integral forms we
  get
  \begin{equation}
    \label{eq:64}
    I := 3 \nu_{0,1} \ellE - (2 \nu_{0,1} + \nu_{0,2}) \ellK =
    \frac{\nu_{0,1} - \nu_{0,2}}{1 + \nu_{0,2}} \int_0^1
    \frac{1 + \nu_{0,2} - 3 t^2}{\sqrt{1 - t^2} \sqrt{1 - f^2 t^2}}
    \dif t
  \end{equation}
  where $f$ is given by \eqref{eq:37}. We have to show that $I$ is not
  negative. This is obviously true for $\nu_{0,2} \ge 2$ because then
  the integrand is positive for $t \in (0,1)$. But we can do even
  better. Assume $\nu_{0,2} < 2$ and denote by $w =
  \sqrt{(1+\nu_{0,2})/3}$ the unique root of the integrand. We split
  the integral \eqref{eq:64} into a positive and a negative part $I =
  I_1 + I_2$ where
  \begin{equation}
    \label{eq:65}
    I_1 = \frac{\nu_{0,1} - \nu_{0,2}}{1 + \nu_{0,2}}
    \int_0^w\frac{1 + \nu_{0,2} - 3 t^2}{\sqrt{1 - t^2} \sqrt{1 - f^2 t^2}} \dif t
  \end{equation}
  and
  \begin{equation}
    \label{eq:66}
    I_2 = \frac{\nu_{0,1} - \nu_{0,2}}{1 + \nu_{0,2}}
    \int_w^1\frac{1 + \nu_{0,2} - 3 t^2}{\sqrt{1 - t^2} \sqrt{1 - f^2 t^2}} \dif t.
  \end{equation}
  Since $f^2 = (\nu_{0,1} - \nu_{0,2})/(\nu_{0,1} (1 + \nu_{0,2}))$ is
  monotone increasing in $\nu_{0,1}$ we have
  \begin{equation}
    \label{eq:67}
    0 \le f^2 \le \lim_{\nu_{0,1} \to \infty} f^2 = \frac{1}{1 + \nu_{0,2}}.
  \end{equation}
  Using this we find
  \begin{equation}
    \label{eq:68}
    1 \le \frac{1}{\sqrt{1 - f^2 t^2}} \le
    \frac{\sqrt{1+\nu_{0,2}}}{\sqrt{1+\nu_{0,2}-t^2}}
    \qquad \text{for}~ t \in [0,1].
  \end{equation}
  Thus, we can estimate
  \begin{equation}
    \label{eq:69}
    \begin{aligned}
      &I_1 \ge \frac{\nu_{0,1}-\nu_{0,2}}{1+\nu_{0,2}}
      \int_0^w\frac{1+\nu_{0,2}-3t^2}{\sqrt{1-t^2}} \dif t
      \quad\text{and}\\
      &I_2 \ge \frac{\nu_{0,1}-\nu_{0,2}}{1+\nu_{0,2}}
      \int_w^1\frac{(1+\nu_{0,2}-3t^2)\sqrt{1+\nu_{0,2}}}{\sqrt{1-t^2}\sqrt{1+\nu_{0,2}-t^2}} \dif t.
    \end{aligned}
  \end{equation}
  This allows to discuss the non-negativity of
  \begin{equation}
    \label{eq:70}
    \int_0^w\frac{1+\nu_{0,2}-3t^2}{\sqrt{1-t^2}} \dif t
    +
    \int_w^1\frac{(1+\nu_{0,2}-3t^2)\sqrt{1+\nu_{0,2}}}{\sqrt{1-t^2}\sqrt{1+\nu_{0,2}-t^2}} \dif t
  \end{equation}
  instead of $I_1 + I_2$. But this is guaranteed by the Theorem's
  assumptions (compare with Equation~\eqref{eq:24}).
\end{proof}

\begin{lemma}
  \label{lem:6}
  Under the assumptions $\nu_{0,2} > v_0$ and \eqref{eq:41}
  ($r_3^2-r_1^2 \ge 0$) we have
  \begin{equation}
    \label{eq:71}
    (2B - A - C) r_3^2 + (A - C) (r_2^2 - r_1^2) \le 0.
  \end{equation}
\end{lemma}

\begin{proof}
  We write~\eqref{eq:71} as
  \begin{multline}
    \label{eq:72}
    (2B - A - C) r_3^2 + (A - C) (r_2^2 - r_1^2)\\
    = (B - A) r_3^2 - (A - C) r_1^2 + (B - C) r_3^2 + (A - C) r_2^2.
  \end{multline}
  The non-positivity of $(B - A) r_3^2 - (A - C) r_1^2$ is shown in
  the proof of Lemma~\ref{lem:7} in~\eqref{eq:75}.  The Lemma's claim
  follows from the fact that the remaining terms are not positive.
\end{proof}

\begin{lemma}
  \label{lem:7}
  Under the assumptions $\nu_{0,2} > v_0$ and \eqref{eq:41}
  ($r_3^2-r_1^2 \ge 0$) we have
  \begin{equation}
    \label{eq:73}
      r_1^2 r_2^2 (C - A) (\nu_{1,1} - \nu_{1,2})
      + r_1^2 r_3^2 (B - C) (1 + \nu_{1,2})
      + r_2^2 r_3^2 (B - A) (1 + \nu_{1,1}) \le 0.
  \end{equation}
\end{lemma}

\begin{proof}
  We write~\eqref{eq:73} as
  \begin{equation}
    \label{eq:74}
    \begin{gathered}
      r_3^2 (r_1^2 (B - C) + r_2^2 (B - A))
      + r_1^2 (r_2^2 (A - C) + r_3^2 (B - C)) \nu_{1,2} \\
      + r_2^2 (-r_1^2 (A - C) + r_3^2 (B - A)) \nu_{1,1}.
    \end{gathered}
  \end{equation}
  The first row is not positive; the second row needs closer
  investigation. We know that $B - A \le A - C$ because of $2A - B - C
  \ge 0$ (see the proof of Lemma~\ref{lem:5}). Therefore we can bound
  the relevant factor of the second row of~\eqref{eq:74} according to
  \begin{equation}
    \label{eq:75}
    -r_1^2 (A - C) + r_3^2 (B - A) \le (A - C) (r_3^2 - r_1^2) \le 0.
  \end{equation}
  This implies that \eqref{eq:74} and therefore \eqref{eq:73} is
  not-positive.
\end{proof}

\section*{Acknowledgments}

The authors gratefully acknowledge support of this research by the
Austrian Science Foundation FWF under grant P21032 (Uniqueness Results
for Extremal Quadrics).

\bibliographystyle{plainnat}
\bibliography{mrabbrev,exquadb}

\end{document}